\documentclass{amsart}
\usepackage[utf8]{inputenc}
\usepackage{amsfonts}
\usepackage{amssymb}
\usepackage{amsmath}
\usepackage{verbatim}
\usepackage{hyperref}
\usepackage{color}
\usepackage{amsthm}
\usepackage{latexsym}
\usepackage{graphicx}
\usepackage{epstopdf}
\usepackage{enumerate}
\usepackage{float}
\usepackage[utf8]{inputenc}
\usepackage{mathtools}
\usepackage{amsfonts}
\usepackage{tabstackengine}
\usepackage{esint}
\usepackage{cite}
\usepackage{verbatim}
\usepackage{euler}
\usepackage{tikz-cd}

\usepackage{enumitem,amssymb}
\newlist{todolist}{itemize}{2}
\setlist[todolist]{label=$\square$}
\usepackage{pifont}

\def\R{\mathbb{R}}
\def\Z{\mathbb{Z}}
\def\N{\mathbb{N}}
\def\Rn{\mathbb{R}^n}

\def\P{\mathcal{P}}
\def\one{{\mathbf 1}}

\newcommand{\norm}[1]{\|{#1}\|}
\newcommand{\set}[1]{\left\{#1\right\}}
\newcommand{\brac}[1]{\left(#1\right)}
\newcommand{\abs}[1]{\left|#1\right|}

\newcommand{\dist}{\mathrm{dist}}

\renewcommand{\d}{\partial}

\newcommand{\jet}{\mathscr{J}}
\newcommand{\RD}{\R^D}

\newtheorem{theorem}{Theorem}[section]
\newtheorem*{acknowledgement}{Acknowledgement}

\newtheorem{definition}[theorem]{Definition}

\newtheorem{lemma}[theorem]{Lemma}
\newtheorem{problem}[theorem]{Problem}

\numberwithin{equation}{section}

\newcommand{\V}{{C^m(\Rn,\RD)}}
	\newcommand{\hV}{{\dot{C}^m(\Rn,\RD)}}
	\newcommand{\snorm}[1]{\norm{#1}_\infty}

\begin{document}
	\title{On the Shape Fields Finiteness Principle}

	\author{Fushuai Jiang}
	\address{Department of Mathematics, UC Davis
		One Shields Ave
		Davis, CA 95616}
	\email{fsjiang@math.ucdavis.edu}

	\author{Garving K. Luli}
	\address{Department of Mathematics, UC     Davis
		One Shields Ave
		Davis, CA 95616}
	\email{kluli@math.ucdavis.edu}

	\author{Kevin O'Neill}
	\address{Department of Mathematics, UC     Davis
		One Shields Ave
		Davis, CA 95616}
	\email{oneill@math.ucdavis.edu}
	
	\maketitle
	
\begin{abstract}
In this paper, we improve the finiteness constant for the finiteness principles for $C^m(\R^n,\R^d)$ and $C^{m-1,1}(\R^n,\R^D)$ selection proven in \cite{FIL16} and extend the more general shape fields finiteness principle to the vector-valued case.
\end{abstract}
	
\section{Introduction}

Suppose we are given integers $m \geq 0$, $n \geq 1$, $D \geq 1$. We write $C^{m}\left( \mathbb{R}^{n},\mathbb{R}^{D}\right) $ to denote the space
of all functions $\vec{F}:\mathbb{R}^{n}\rightarrow \mathbb{R}^{D}$ whose
derivatives $\partial^{\alpha }\vec{F}$ (for all $\left\vert \alpha \right\vert
\leq m$) are continuous and bounded on $\mathbb{R}^{n}$, equipped with the norm
	\begin{equation*}
	\norm{\vec{F}}_\V := \max_{\abs{\alpha} \leq m}\sup_{x \in \Rn} \norm{\d^\alpha \vec{F}(x)}_\infty=\max_{\substack{\abs{\alpha} \leq m\\ 1\le j\le D}}\sup_{x \in \Rn} |\d^\alpha F_j(x)|.
	\end{equation*}
	Here and below, we view $ \d^\alpha \vec{F}(x) = (\d^\alpha F_1(x), \cdots, \d^\alpha F_D(x)) $ as a vector in $ \RD $.
	
	We write $ \hV $ to denote the vector space of $ m $-times continuously differentiable $ \RD $-valued functions whose $ m $-th order derivatives are bounded, equipped with the seminorm
	\begin{equation*}
	\norm{\vec{F}}_\hV := \max_{\abs{\alpha} = m}\sup_{x \in \Rn} \norm{\d^\alpha \vec{F}(x)}_\infty.
	\end{equation*}

We write $%
C^{m-1,1}\left( \mathbb{R}^{n},\mathbb{R}^{D}\right) $ to denote the space of
all $\vec{F}:\mathbb{R}^{n}\rightarrow \mathbb{R}^{D}$ whose derivatives $\partial
^{\alpha}\vec{F}$ (for all $\left\vert \alpha \right\vert \leq m-1$) are bounded
and Lipschitz on $\mathbb{R}^{n}$. When $D=1$, we write $C^{m}\left( \mathbb{%
R}^{n}\right) $ and $C^{m-1,1}\left( \mathbb{R}^{n}\right) $ in place of $%
C^{m}\left( \mathbb{R}^{n},\mathbb{R}^{D}\right) $ and $C^{m-1,1}\left(
\mathbb{R}^{n},\mathbb{R}^{D}\right) $.

We write $\vec{\P} $ to denote the vector space $\bigoplus_{j =1}^D\P$, where $\P$ is the space of polynomials on $\Rn$ with degree no greater than $m-1$. Note that $\dim\vec{\P}=D\cdot \binom{n+m-1}{m-1}$.

Quantities $c\left( m,n\right) $, $C\left( m,n\right) $, $k\left(
m,n\right) $, etc., denote constants depending only on $m$, $n$; these
expressions may denote different constants in different occurrences. Similar
conventions apply to constants denoted by $C\left( m,n,D\right)$, $k\left(m,n,
D\right)$, etc.

If $S$ is any finite set, then $\abs{S}$ denotes the number of
elements in $S$.

Let $E \subset \mathbb{R}^n$ be given. Suppose at each $x \in E$, we are given a convex set $K(x) \in \mathbb{R}^D$. A \underline{selection} of $(K(x))_{x \in E}$ is a map $\vec{F}: \mathbb{R}^n \rightarrow \mathbb{R}^D$ such that $\vec{F}(x) \in K(x)$ for all $x \in E$.

We are interested in the following selection problem. 

\begin{problem}\label{problem}
Let $E \subset \mathbb{R}^n$. For each $x \in E$, suppose we are given a convex $K(x)\subset \mathbb{R}^n$. Given a number $M > 0$, how can decide if there exists a selection $\vec{F} \in C^{m-1,1}(\mathbb{R}^n,\mathbb{R}^D)$ or $\vec{F} \in C^m(\mathbb{R}^n,\mathbb{R}^D)$ with $\left\Vert \vec{F}\right\Vert _{C^{m}\left( \mathbb{R}^{n},\mathbb{R}%
^{D}\right) }\leq C^{\#}M$ or $\left\Vert \vec{F}\right\Vert _{C^{m-1,1}\left( \mathbb{R}^{n},%
\mathbb{R}^{D}\right) }\leq C^{\#}M$, where $C^{\#}$ depends only on $m,n,D$?
\end{problem}

In \cite{FIL16}, the authors addressed Problem \ref{problem} by proving the following

\begin{theorem}[Finiteness Principle for Smooth Selection]
\label{th1}For large enough ${k^{\sharp}}=k\left( m,n,D\right) $ and $%
C^{\#}=C\left( m,n,D\right) $, the following hold.

\begin{description}
\item[(A) $C^{m}$ FLAVOR] Let $E\subset \mathbb{R}^{n}$ be finite. For each $%
x\in E$, let $K\left( x\right) \subset \mathbb{R}^{D}$ be convex. Suppose
that for each $S\subset E$ with $\abs{S} \leq {k^{\sharp}}$, there
exists $\vec{F}^{S}\in C^{m}\left( \mathbb{R}^{n},\mathbb{R}^{D}\right) $ with
norm $\left\Vert \vec{F}^{S}\right\Vert _{C^{m}\left( \mathbb{R}^{n},\mathbb{R}%
^{D}\right) }\leq 1$, such that $\vec{F}^{S}\left( x\right) \in K\left( x\right) $
for all $x\in S$. \newline
Then there exists $\vec{F}\in C^{m}\left( \mathbb{R}^{n},\mathbb{R}^{D}\right) $
with norm $\left\Vert \vec{F}\right\Vert _{C^{m}\left( \mathbb{R}^{n},\mathbb{R}%
^{D}\right) }\leq C^{\#}$, such that $\vec{F}\left( x\right) \in K\left( x\right) $
for all $x\in E$.

\item[(B) $C^{m-1,1}$ FLAVOR] Let $E\subset \mathbb{R}^{n}$ be arbitrary.
For each $x\in \mathbb{R}^n$, let $K\left(
x\right) \subset \mathbb{R}^{D}$ be a closed convex set. Suppose that for
each $S\subset E$ with $\abs{S} \leq {k^{\sharp}}$, there
exists $\vec{F}^{S}\in C^{m-1,1}\left( \mathbb{R}^{n},\mathbb{R}^{D}\right) $ with
norm $\left\Vert \vec{F}^{S}\right\Vert _{C^{m-1,1}\left( \mathbb{R}^{n},\mathbb{R}%
^{D}\right) }\leq 1$, such that $\vec{F}^{S}\left( x\right) \in K\left( x\right) $
for all $x\in S$.\newline
Then there exists $\vec{F}\in C^{m-1,1}\left( \mathbb{R}^{n},\mathbb{R}^{D}\right)
$ with norm $\left\Vert \vec{F}\right\Vert _{C^{m-1,1}\left( \mathbb{R}^{n},%
\mathbb{R}^{D}\right) }\leq C^{\#}$, such that $\vec{F}\left( x\right) \in K\left(
x\right) $ for all $x\in \mathbb{R}^{n}$.
\end{description}
\end{theorem}

  Therefore, Theorem \ref{th1} tells us when there exists a $C^{m-1,1}$ selection $\vec{F}$ of $(K(x))_{x \in E}$ for the case of infinite $E$ and provides estimates for the $C^m$-norm of a selection for finite $E$. 
  
  Theorem \ref{th1} for the case $D = 1$ and $K(x)$ being a singleton for each $x \in E$ was conjectured by Y. Brudnyi and P. Shvartsman in \cite{BS94-Tr}. 

The number ${k^{\sharp}}$ in Theorem \ref{th1} is called the \underline{finiteness number}. The ${k^{\sharp}}$ obtained in \cite{FIL16} is ${k^{\sharp}} = 100+(D+2)^{l_{*}+100}$, where $l_{*}= \binom{m+n}{n}$.

Here, we give a sharper bound on ${k^{\sharp}}$. Our first result is the following.
	
	\begin{theorem}\label{thm:smooth selection improved}
		The $k^\sharp$ found in Theorem \ref{th1} may be taken to be $k^\sharp=2^{\dim \vec{\P}}$, where $\dim\vec{\P}=D\cdot \binom{n+m-1}{m-1}$. 
	\end{theorem}
	
	A few remarks on Theorem \ref{thm:smooth selection improved} are in order. For $D=1$, our result coincides with the one proven by P. Shvartsman \cite{Shv08}. A similar result was obtained by E. Bierstone and P. Milman \cite{BM07} in the case where each $K(x)$ consists of a single point, corresponding to the finiteness principle proven by C. Fefferman in \cite{F05-Sh}. Our present approach is inspired by \cite{BM07}.
	
	In the case $D=1$ and $m=2$, Theorem \ref{thm:smooth selection improved} gives $k^\sharp=4\cdot 2^{n-1}$. This is comparable to the finiteness constant $3\cdot 2^{n-1}$ given by Shvartsman\cite{Shv87}, which he shows to be optimal. See also \cite{BS01}.
	
    To prove Theorem \ref{thm:smooth selection improved}, we will need to extend the finiteness principle proven in \cite{FIL16} to the vector valued case. 	
	
	\begin{theorem}\label{thm:main theorem}
		\newcommand{\X}{\mathbb{X}}
		The following holds for $ \X = \V $ and $ \X = \hV $.
		
		Let $S\subset \R^n$ be a finite set of diameter at most 1. For each $x\in S$, let $\vec{G}(x)\subset\vec{\P}$ be convex. Suppose that for every subset $S' \subset S$ with $\abs{S} \leq 2^{\dim \Vec{\P}}$, there exists $F^{S'}\in \X$ such that $\|F^{S'}\|_{\X}\le 1$ and $\jet_x F^{S'}\in \vec{G}(x)$ for all $x\in S'$.
		
		Then, there exists $F\in \X$ such that $\|F\|_{\X}\le \gamma$ and $\jet_xF\in \vec{G}(x)$ for all $x\in S$.
		
		Here, $\gamma$ depends only on $m,n,D$, and $|S|$.
	\end{theorem}

Because the constant $\gamma$ depends on the number of points in $S$, following \cite{Shv08}, we will refer to Theorem \ref{thm:main theorem} as a ``weak finiteness principle". 

To conclude the introduction, we give an overview of how we prove Theorems \ref{thm:smooth selection improved} and \ref{thm:main theorem}. The proof of Theorem \ref{th1} given in \cite{FIL16} is via a more general finiteness principle for shape fields, see Theorem \ref{thm:SF FP} below. Using Theorem \ref{thm:main theorem}, we will show an improved bound for ${k^{\sharp}}$ in the finiteness principle for shape fields (i.e., Theorem \ref{thm:SF FP}); we can then deduce the bound for ${k^{\sharp}}$ in Theorem \ref{th1}, obtaining the bound asserted in Theorem \ref{thm:smooth selection improved}. The heart of the matter therefore lies in Theorem \ref{thm:main theorem}. To put things in perspective, we would like to point out that one can't directly apply the techniques from \cite{BM07} because of the nonlinear structure in the selection problem and that the result in \cite{Shv08} is for scalar-valued functions. To prove our main theorem (Theorem \ref{thm:main theorem}), we will adapt the strategy from \cite{BM07} with some new ingredients: Instead of linear structure, we will handle general convex structure using the duality theorem of linear programming to describe the relevant convex sets.

This paper is part of a literature on extension, interpolation, and
selection of functions, going back to H. Whitney's seminal work \cite{W34-1,W34-2,W34-3}, and including fundamental contributions by G. Glaeser \cite{G58}, Y. Brudnyi and P. Shvartsman \cite{Shv01,Shv02,BS01,BS85,BS94-Tr,BS94-W,BS97,BS98,Shv87}, and E. Bierstone, P. Milman,
and W. Paw{\l}ucki \cite{BM07,BMP03,BMP06}, and C. Fefferman \cite{F05-J,F05-L,F05-Sh,F06,F07-L,F07-St,F09-Data-3,F09-Int,F10,FShv18}.

\begin{acknowledgement}
We are indebted to Pavel Shvartsman for his valuable comments. 

The first author is supported by the UC Davis Summer Graduate Student Researcher Award
and the Alice Leung Scholarship in Mathematics. The second author is supported by NSF Grant
DMS-1554733 and the UC Davis Chancellor’s Fellowship
\end{acknowledgement}

\section{Background and main results}

	\subsection{Polynomial and Whitney fields}
	\newcommand{\ring}{\mathcal{R}}
	\newcommand{\vP}{{\Vec{P}}}
	\newcommand{\PD}{{\vec{\P}}}
	\newcommand{\M}{\mathcal{M}}
	
	We write $ \P $ to denote the vector space of polynomials on $ \Rn $ with degree no greater than $ m-1 $.
	
	For $ x \in \Rn $, let $ F $ be $ (m-1) $-times differentiable at $ x $. We identify the $ (m-1) $-jet of $ F $ at $ x $ with the $ (m-1)^{\text{st}} $-degree Taylor polynomial of $ F $ at $ x $:
	\begin{equation*}
	\jet_x F(y) := \sum_{\abs{\alpha} \leq m-1}\frac{\d^\alpha F(x)}{\alpha!}(y-x)^\alpha.
	\end{equation*}
	For $ P, Q \in \P $ and $ x \in \Rn $, we define
	\begin{equation*}
	P \odot_x Q := \jet_x(PQ).
	\end{equation*}
	The operation $ \odot_x $ turns $ \P $ into a ring, which we denote by $ \ring_x $.

	We define
	\begin{equation*}
	\PD := \underbrace{\P \oplus \cdots \oplus \P}_{\text{$ D $ copies}}.
	\end{equation*}
	Thus, every $ \vP \in \PD $ has the form $ \vP = (P_1, \cdots, P_D) $, with $ P_j \in \P $ for $ j = 1, \cdots, D $.
	
	Let $ \vec{F} = (F_1, \cdots, F_D) $ be a $ \RD $-valued function $ (m-1) $-times differentiable at $ x \in \Rn $. We define
	\begin{equation*}
	\jet_x\vec{F} := (\jet_x F_1, \cdots, \jet_x F_D) \in \PD.
	\end{equation*}
	
	We will also use the $ \ring_x $-module structure on $ \PD $, whose multiplication is given by
	\begin{equation*}
	R\odot_x \vP := (R\odot_x P_1, \cdots, R\odot_x P_D) \in \PD,
	\end{equation*}
	for $ x \in \Rn $, $ \vP = (P_1, \cdots, P_D) \in \PD $, and $ R \in \ring_x $.

	Let $S \subset \Rn $ be a finite set. A \underline{Whitney field} is an array $ (\vec{P}^x)_{x \in S} $ parameterized by points in $ S $, where $ \vP^x \in \PD $ for $ x \in S $. We write $ W^m(S) $ to denote the space of Whitney fields on $S$.
	
	Given $ (\vP^x)_{x \in S} \in W^m(S) $, we define
	\begin{equation}
	\norm{(\vP^x)_{x \in S}}_{W^m(S)}:= \max_{\substack{x \in S\\\abs{\alpha} \leq m-1}} \snorm{\d^\alpha \vP^x(x)} + \max_{\substack{x,y \in S, x \neq y\\\abs{\alpha} \leq m-1}} \frac{\snorm{\d^\alpha (\vP^x - \vP^y)(x)}}{\abs{x - y}^{m-\abs{\alpha}}}.
	\label{eq.W-1}
	\end{equation}
	Note that $ \norm{\cdot}_{W^m(S)} $ is a norm on $ W^m(S) $.
	
	We will also be using the seminorm
	\begin{equation}
	\norm{(\vP^x)_{x \in S}}_{\dot{W}^m(S)}:=  \max_{\substack{x,y \in S, x \neq y\\\abs{\alpha} \leq m-1}} \frac{\snorm{\d^\alpha (\vP^x - \vP^y)(x)}}{\abs{x - y}^{m-\abs{\alpha}}}.
	\label{eq.W-2}
	\end{equation}
	
	We use $ \PD^* $ to denote the dual of $ \PD $. We use $ W^m(S)^* $ to denote the dual of $ W^m(S) $. An element $ \xi \in W^m(S)^* $ has the form $ \xi = (\xi_x)_{x \in S} $. We use $ \xi[(\vP^x)_{x \in S}] $ to denote the action of $ \xi \in W^m(S)^* $ on $ (\vP^x)_{x \in S} \in W^m(S) $.

	\begin{lemma}\label{lem.WT}
		Fix $m,n,D\in\N$. There exists $C=C(m,n)<\infty$ such that the following hold.
		\begin{enumerate}
			\item Let $ S \subset \Rn $ be a finite set.
			\begin{enumerate}
				\item For all $ \vec{F} \in \V $, $ \norm{(\jet_x \vec{F})_{x \in S}}_{{W}^m(S)} \leq C\norm{\vec{F}}_\V $.
				\item Let $C^m_{loc}(\Rn,\RD)$ denotes the space of $\RD$-valued functions on $\Rn$ with continuous derivatives up to order $m$. For all $ \vec{F} \in C^m_{loc}(\Rn,\RD) $, $ \norm{(\jet_x \vec{F})_{x \in S}}_{\dot{W}^m(S)} \leq C\norm{\vec{F}}_\hV $. 
			\end{enumerate}
			
			\item Let $ S \subset \Rn $ be a finite set. There exists a linear map $ T_w^S : W^m(S) \to \V $ such that
			\begin{enumerate}
				\item $ \norm{T_w^S[(\vP^x)_{x \in S}]}_\V \leq \norm{(\vP^x)_{x \in S}}_{W^m(S)} $,
				\item $ \norm{T_w^S[(\vP^x)_{x \in S}]}_\hV \leq \norm{(\vP^x)_{x \in S}}_{\dot{W}^m(S)} $, and
				\item $ \jet_x \circ T_w^S[(\vP^x)_{x \in S}] = \vP^x $ for each $ x \in S $.
			\end{enumerate}
			
		\end{enumerate}
	\end{lemma}

	Lemma \ref{lem.WT}(1) is simply Taylor's theorem. Lemma \ref{lem.WT}(2) is an $\RD$-valued version of the Whitney Extension Theorem (for finite sets), the proof of which follows the $\R$-valued case by component-wise treatment. The proof for the classical Whitney Extension Theorem can be found in e.g. \cite{St79}.

	\subsection{Shape fields}
	
	In this section, we generalize a key object introduced in \cite{FIL16}.

	\begin{definition}\label{def.shape field}
		\newcommand{\vG}{{\Vec{\Gamma}}}
		
		Let $S \subset \Rn$ be finite. For each $x \in S$, $0 < M < \infty$, let $\vG(x,M) \subset \vP$ be a (possibly empty) convex set. We say that $(\vG(x,M))_{x \in S, M > 0}$ is a \underline{vector-valued shape field} if for all $x \in S$ and $ 0 < M' \leq M < \infty $, we have $ \vG(x,M') \subset \vG(x,M) $.
		
		When $ D = 1 $, we write $ \Gamma(x,M) $ instead of $ \vG(x,M) $, and we omit the adjective ``vector-valued''.
		
	\end{definition}
	
	%

	\begin{definition}\label{def.SF}
		\newcommand{\vG}{{\Vec{\Gamma}}}
		Let $ C_w, \delta_{\max} $ be positive real numbers. We say that a vector-valued shape field $ (\vG(x,M))_{x \in S, M > 0}$ is \underline{$ (C_w, \delta_{\max}) $-convex} if the following condition holds:
		
		Let $ 0 < \delta \leq \delta_{\max} $, $ x \in S $, $ 0 < M < \infty $, $ \vP_1, \vP_2 \in \PD $, $ Q_1, Q_2 \in \P $. Assume that
		\begin{enumerate}
			\item $ \vP_1, \vP_2 \in \vG(x,M) $;
			\item $ \snorm{\d^\alpha (\vP_1 - \vP_2)(x)} \leq M\delta^{m-\abs{\alpha}} $ for $ \abs{\alpha} \leq m-1 $;
			\item $ \abs{\d^\alpha Q_i(x)} \leq \delta^{-\abs{\alpha}} $ for $ \abs{\alpha} \leq m-1 $ and $ i = 1,2 $;
			\item $ Q_1 \odot_x Q_1 + Q_2 \odot_x Q_2 = 1 $.\\
			Then
			\item $ \vP := \sum_{i = 1}^2 (Q_i \odot_x Q_i) \odot_x \vP_i
			\in \vG(x,C_wM) $. Here, in each summand, the first multiplication is the ring multiplication in $ \ring_x $, and the second is the action of $ \ring_x $ on the $ \ring_x $-module.
		\end{enumerate}
		
	\end{definition}

	\subsection{Main technical results}

	The main technical results are the following two theorems. The first is the Finiteness Principle for vector-valued shape fields, and second improves the finiteness constant.
	
	\begin{theorem}\label{thm:SF FP}
		There exists $k^{\sharp} = k^{\sharp}(m,n,D)$ such that the following holds.
		
		Let $E\subset \R^n$ be a finite set and $(\vec{\Gamma}(x,M))_{x\in E,M>0}$ be a $(C_w,\delta_{\max})$-convex vector-valued shape field. Let $Q_0\subset\R^n$ be a cube of side length $\delta_{Q_0}\le\delta_{\max}$ and $x_0\in E\cap 5Q_0$ and $M_0>0$ be given.
		
		Suppose that for each $S\subset E$ with $|S|\le k^\sharp$, there exists a Whitney field $(\vec{P}^z)_{z\in S}$ such that
		\begin{equation}\label{eq:SF homogeneity condition}
		\|(\vec{P}^z)_{z\in S}\|_{\dot{W}^m(S)}\le M_0
		\end{equation}
		
		and
		\begin{equation}
		\vec{P}^z\in \vec{\Gamma}(z,M_0)\text{ for all }z\in S.
		\end{equation}
		
		Then, there exist $\vec{P}^0\in\vec{\Gamma}(x_0,M_0)$ and $\vec{F}\in C^m(Q_0,\R^D)$ such that
		\begin{itemize}
			\item $J_z\vec{F}\in \vec{\Gamma}(z,CM)$ for all $z\in E\cap 5Q_0$.
			\item $\snorm{\d^\alpha(\vec{F}-\vec{P}^0)(x)}\le CM_0\delta_{Q_0}^{m-|\alpha|}$ for all $x\in Q_0, |\alpha|\le m$.
			\item In particular, $\snorm{\d^\alpha\vec{F}(x)}\le CM_0$ for all $x\in Q_0, |\alpha|=m$.
		\end{itemize}
	\end{theorem}

	The case of scalar-valued shape fields ($D=1$) was proven in \cite{FIL16}. In this paper, we will use the $D=1$ case to prove the more general Theorem \ref{thm:SF FP} stated above using a gradient trick, inspired by \cite{FIL16, FL14}.
	
	\begin{theorem}\label{thm:SF FP improved}
		One may take $ k^{\sharp} = 2^{\dim \vec{\P}} $ in Theorem \ref{thm:SF FP}.
	\end{theorem}
	
	\begin{proof}[Proof of Theorem \ref{thm:SF FP improved} via Theorem \ref{thm:main theorem}]
		Take as given the hypotheses for Theorem \ref{thm:SF FP}, but with $k^\sharp=2^{\dim\vec{\P}}$. This means that for each $S'\subset E$ with $|S'|=2^{\dim\vec{P}}$, there exists $(\vec{P}^z)_{z\in S'}$ such that
		\begin{equation}
		\|(\vec{P}^z)_{z\in S'}\|_{\dot{W}^m(S')}\le M_0
		\end{equation}
		and
		\begin{equation}
		\vec{P}^z\in\vec{\Gamma}(z,M_0).
		\end{equation}
		
		Recall that in the definition of shape field, we require $\Gamma(x,M)$ be convex for all $x\in S$ and $M>0$.
		
		
		Let $S\subset E$ with $|S|\le k^\sharp$, where $k^\sharp$ is as initially stated in Theorem \ref{thm:SF FP} (and coming from \cite{FIL16} and our gradient trick for $D\ge2$). Then, the above holds for all $S'\subset S$ with $|S'|=2^{\dim\vec{\P}}$, so by the homogeneous version of Theorem \ref{thm:main theorem}, there exists $\vec{F}\in\dot{C}^m(\R^n,\R^D)$ such that
		\begin{equation}\label{eq:control on function norm}
		\|\vec{F}\|_{\dot{C}^m(\R^n,\R^D)}\le\gamma M_0
		\end{equation}
		and
		\begin{equation}
		J_x\vec{F}\in \Gamma(x,M_0)\text{ for all }x\in S.
		\end{equation}
		By \eqref{eq:control on function norm}, we have
		\begin{equation}
		\|(J_x\vec{F}^x)_{x\in S}\|_{\dot{W}^m(S)}\le C\gamma M_0.
		\end{equation}
		Thus, the hypotheses for Theorem \ref{thm:SF FP} with the $k^\sharp$ from the initial statement are satisfied.
	\end{proof}
	
	At this point, we have shown that the shape fields finiteness principle holds with an improved value of $k^\sharp$ (Theorem \ref{thm:SF FP improved}); the next step is to show that the selection problem of Theorems \ref{thm:smooth selection improved} and \ref{thm:main theorem} may be described through shape fields.
	
	\begin{proof}[Proof of Theorem \ref{thm:smooth selection improved} via Theorem \ref{thm:SF FP improved}]
		Let \begin{equation}
		\vec{\Gamma}(x,M):=\set{\vec{P}\in\vec{\P}:\snorm{\d^\alpha\vec{P}(x)}\le M, \vec{P}(x)\in K(x)}.
		\label{eq.3.8}
		\end{equation}
		It suffices to observe that $(\vec{\Gamma}(x,M))_{x\in E,M>0}$ is a $(C,1)$-convex shape field when $K(x)$ is convex for each $ x \in E $.

		\newcommand{\vG}{\Vec{\Gamma}}
		\newcommand{\ox}{{\odot_x}}
		Let $ \delta \in (0,1] $, $ x \in E $, $ M \in (0,\infty) $, $ \vP_1, \vP_2 \in \PD $, and $ Q_1, Q_2 \in \P $ be given, such that
		\begin{enumerate}[label = (C\arabic*)]
			\item $ \vP_1, \vP_2 \in \vG(x,M) $ with $ \vG(x,M) $ as in \eqref{eq.3.8};
			\item $ \snorm{\d^\alpha (\vP_1 - \vP_2)(x)} \leq M\delta^{m-\abs{\alpha}} $ for $ \abs{\alpha} \leq m-1 $;
			\item $ \abs{\d^\alpha Q_i(x)} \leq \delta^{-\abs{\alpha}} $ for $ \abs{\alpha} \leq m-1 $, $ i = 1,2 $; and
			\item $ Q_1\ox Q_1 + Q_2 \ox Q_2 = 1 $.
		\end{enumerate}
		We set
		\begin{equation*}
		\vP:= \sum_{i = 1,2}Q_i\ox Q_i \ox \vP_i.
		\end{equation*}
		We want to show that $ \vP \in \vG(x,CM) $ for some $ C = C(m,n,D) $.
		
		It is clear from (C1) and (C4) that $ \vP(x) = K(x) $. It remains to show that $ \snorm{\d^\alpha\vP(x)} \leq CM $.
		
		Using the product rule, we have, for $ \abs{\alpha} \leq m-1 $,
		\begin{equation*}
		\begin{split}
		\d^\alpha \vP(x)
		&= \sum_{i = 1,2} \sum_{\beta \leq \alpha}\sum_{\gamma\leq \beta}
		C_{\alpha,\beta,\gamma}\cdot
		\d^\gamma Q_i(x)
		\cdot
		\d^{\beta-\gamma}Q_i(x)
		\cdot
		\d^{\alpha-\beta}\vP_i(x).
		\end{split}
		\end{equation*}
		By (C4), we have $ \d^\alpha(Q_2\ox Q_2) = -\d^\alpha(Q_1\ox Q_1) $ for $ \abs{\alpha} > 0 $. It follows from (C2) and (C3) that $ \snorm{\d^\alpha \vP(x)} \leq CM $.

		
	\end{proof}

	Thus, it remains to establish Theorem \ref{thm:main theorem}. This will be done in Section \ref{sec:main proof}

	\section{Whitney norm and dual norm on clusters}
	\label{sect:W-norm}
	
	In this section, we review the data structure in \cite{BM07}, and prove a series of results that allows us to reduce the size of supports for linear functionals on $ W^m(S)^* $.
	
	\newcommand{\diam}{\mathrm{diam}}

	We write $ \abs{S} $ to denote the cardinality of a finite set $ S \subset \Rn $.

	If $ X,Y \subset \Rn $, we define
	\begin{equation*}
	\begin{split}
	\diam(X) &:= \max_{x,x' \in X}\abs{x - x'}\text{ and }\\
	\dist(X,Y) &:= \min_{x\in X, y \in Y} \abs{x-y}.
	\end{split}
	\end{equation*}

	A rooted tree (``tree'' for short) is an undirected graph with a distinct node (i.e., the root) in which any two nodes are connected by exactly one path. A leaf of a tree is any non-root node of degree one.

	\newcommand{\tree}{\mathscr{T}}

	Let $ S \subset \Rn $ be a finite set. We consider trees $ \tree $, each node of which corresponds to a subset of $S$, that satisfy the following properties.
	\begin{enumerate}[label = (T\arabic*)]
		\item\label{tree-1}The root of $ \tree $, $ R(\tree) = S $.
		\item\label{tree-2}If $ V $ is a node, then $ V $ corresponds to a subset of $ S $. The children of any node $ V $ form a partition of $ V $ (unless $ V $ is a leaf).
		\item\label{tree-3}The nodes of any given level correspond to a partition of $ S $. The leaves of $ \tree $ are the singletons $ \set{x} $, with $ x \in S $.
		\item\label{tree-4}The number of nodes of level $ \ell = 0,1,\cdots $ is a strictly increasing function of $ \ell $.
	\end{enumerate}

	A collection of points
	\begin{equation*}
	\underline{x} = \set{x_V \in S : V \in \tree \setminus leaves(\tree)}
	\end{equation*}
	is called a \underline{set of reference points} for $ \tree $ if, for each $ V $, $ x_V \in V $ and $ x_V = x_W $ for some child $ W $ of $ V $. We adopt the convention $ x_{\set{x}} := x $ in the last level.
	
	Let $ \underline{x} $ be a set of reference points of $ \tree $. For each $ V \in \tree \setminus leaves(\tree) $, define
	\begin{equation*}
	V(\underline{x}) := \set{x_W : W \text{ is a child of } V}\,.
	\end{equation*}
	
	Suppose $ x \in S \setminus \set{x_S} $. Then there is a unique node $ V $ of highest level such that $ x \in V \setminus \set{x_V} $. We set
	\begin{equation}
	ref(x) := x_V.
	\label{eq.ref-x}
	\end{equation}
	We also set
	\begin{equation}
	U(x) := \text{ the oldest ancestor of } U \text{ such that } x = x_U.
	\label{eq.U-x}
	\end{equation}
	
	A \underline{trunk} $ T $ of $ \tree $ denotes a directed path from the root $ S $ to level $ height(\tree)-1 $. Let $ T $ be a trunk of $ \tree $. We define the set of \underline{branch nodes} $ B(T) $ as the set of nodes of $ \tree $ which are adjacent to $ T $.
	
	Recall that the height of a tree is defined to be the number of edges on the longest downward path between the root and a leaf. We define the notion of ``clustering'' as follows.
	
	\begin{definition}\label{def.clustering}
		Let $ S \subset \Rn $ be finite. Let $ \tree $ be a tree of subsets of $ S $ that satisfies {\rm\ref{tree-1}} to {\rm{\ref{tree-4}}}. We say that $ \tree $ is a \underline{clustering} of $ S $ if $ \tree $ has a set of reference points $ \underline{x} = \set{x_V} $ such that for each $ l = 0, 1, \cdots, height(\tree) - 1 $, the set
		\begin{equation*}
		\Pi := \set{V(\underline{x}) : level(V) = \ell}
		\end{equation*}
		forms a partition of
		\begin{equation*}
		\set{a_W : level(W) = \ell+1}
		\end{equation*}
		satisfying
		\begin{equation}
		\begin{split}
		\abs{x - y} &\geq c_{\underline{x}}\cdot \diam(S) \text{ for each } S \in \Pi \text{, } x \neq y \text{ in } S\text{, and}
		\\
		\dist(S,S') &\geq c_{\underline{x}} \cdot \diam(S) \text{ for all } S, S' \in \Pi, S \neq S'.
		\end{split}
		\label{eq.clust}
		\end{equation}
		Here, $ 0 < c_{\underline{x}}(n,\abs{S}) \leq 1 $ is called the \underline{clustering constant}.
	\end{definition}
	
	\newcommand{\clu}{\mathcal{C}}
	We write $ \clu = \clu(\tree,\underline{x}) $ to denote a clustering $ \tree $ of $ S $ together with a set of reference points $ \underline{x} $.

	\begin{lemma}[Lemma 2.4 of \cite{BM07}]\label{lem.clustering-exist}
		Given a finite set $ S \subset \Rn $, we can always find a clustering $ \tree $ of $ S $ such that for any set of reference points $ \underline{x} $ for $ \tree $, the condition of Definition \ref{def.clustering} is satisfied with some $ 0 < c_{\underline{x}} \leq 1 $, where $ c_{\underline{x}} $ depends only on $ m,n, $ and $ \#(S) $.
	\end{lemma}

	\newcommand{\trinorm}[1]{|||#1|||}
	
	\begin{definition}\label{def.clu-norm}
		Let $ \clu = \clu(\tree,\underline{x}) $ be a clustering of $ S $ with a set of reference points $\underline{x}$. We define the $ C^m $-clustering norm $ \norm{\cdot}_{\clu} $ on $ W^m(S) $ to be
		\begin{equation*}
		\norm{(\vP^x)_{x \in S}}_{\clu} := \max\set{ \trinorm{(\vP^x)_{x \in S}}_{\clu}\,\,,\,\,\norm{\vP^{x_S}}_{x_S}}\,,
		\end{equation*}
		where
		\begin{equation*}
		\trinorm{\vec{P}}_{\clu} := \max_{\substack{x \in S \setminus\set{x_S}\\
				y = ref(x)\\
				\abs{\alpha}\leq m
		}}
		\frac{\snorm{\d^\alpha (\vP^x - \vP^y)(x)}}{\abs{x-y}^{m-\abs{\alpha}}}
		\,\,\text{ and }\,\,
		\norm{\vP^{x_S}}_{x_S} :=  \max_{\abs{\alpha}\leq m}\snorm{\d^\alpha \vP^{x_S}(x_S)}
		\end{equation*}
	\end{definition}

	\begin{lemma}[Proposition 3.2 of \cite{BM07}]\label{lem.clust-equiv}
		Let $ S \subset \Rn $ be a finite set, and let $ \clu = \clu(\tree,\underline{x}) $ be a clustering of $ S $ with a set of reference points $\underline{x}$ and clustering constant $ c_{\underline{x}} $. Then
		\begin{equation}
		\norm{(\vec{P}^x)_{x \in S}}_{W^m(S)} \leq C \norm{(\vec{P}^x)_{x \in S}}_{\clu}.
		\label{eq.lem.3.9}
		\end{equation}
		Here, $ C = C(c_{\underline{x}},m,n,\abs{S},B) $, where $ B $ is an upper bound on $ \diam(S) $.
	\end{lemma}

	Next we characterize linear functionals on clusters.

	Let $ S \subset \Rn $ be finite, and let $ \xi = (\xi_x)_{x \in S} \in W^m(S)^* $.
	
	Let $ \clu(\tree,\underline{x}) $ be a clustering of $ S $. For each node $ V \in \tree $, we define $ \xi_V $ by the formula
	\begin{equation}
	\xi_V := \sum_{x \in V} \xi_x\,.
	\label{eq.xi-V}
	\end{equation}
	
	\begin{lemma}[Lemma 5.1 of \cite{BM07}]
	    Let $ S \subset \Rn $ be a finite set, and let $ \clu = \clu(\tree,\underline{x}) $ be a clustering of $ S $ with a set of reference points $\underline{x}$ and clustering constant $ c_{\underline{x}} $. Let $ref(x)$ and $U(x)$ be as in \eqref{eq.ref-x} and \eqref{eq.U-x}, respectively.
	    The action of $ \xi \in W^m(S)^* $ has the form:
	\begin{equation}
	\xi[(\vec{P}^x)_{x \in S}] = \sum_{x \in S \setminus \set{x_S}} \xi_{U(x)}(\vP^x - \vP^{ref(x)}) + \xi_S(\vP^{x_S})\,.
	\label{eq.action}
	\end{equation}
	\end{lemma}

	As a consequence, we can compute the cluster dual norm using the formula:
	\begin{equation}
	\begin{split}
	\norm{\xi}_{\clu^*} &=
	\sum_{\substack{x \in S \setminus \set{x_S}\\\abs{\alpha}\leq m\\1\le j\le D}} \abs{x - ref(x)}^{m-\abs{\alpha}}\abs{\xi_{U(x)}\brac{ 0,...,0,\frac{(\cdot - x)^\alpha}{\alpha !},0,...,0 }} \\
	&\,\,\,\,\,\,\,\,\,\,\,\,\,\,\,\,\,\,\,\,
	+ \sum_{\substack{\abs{\alpha}\leq m\\1\le j\le D}}\abs{\xi_{S}\brac{0,...,0,\frac{(\cdot - x_S)^{\alpha}}{\alpha!},0,...,0}}\,.
	\end{split}
	\label{eq.cluster-dual}
	\end{equation}
	
	In the above, the nontrivial expression in the arguments of $\xi_S$ and $\xi_{U(x)}$ are in the $j$-th coordinates.

	\begin{lemma}\label{lem.red-1}
		Let $ S \subset \Rn $ be a finite set, and let $ \Phi : S \times \PD^* \to \R $ be a function that is positively homogeneous with degree one on the fibers and vanishes along the zero section. Let $ \tree $ be a clustering of of $ S $. Let $ \xi \in W^m(S)^* $. For each $ V \in \tree $, define $ \xi_V $ as in \eqref{eq.xi-V}. Define
		\begin{equation*}
		\Phi(\xi_V) :=\sum_{x \in V}\Phi_x(\xi_x)\,,
		\end{equation*}
		and set $ \bar{\xi}_V := (\xi_V, \Phi(\xi_V)) \in \PD^* \oplus \R $. Let $ T $ be a trunk of $ \tree $, and let $ \Xi(T) $ denote the linear span of $ \set{\bar{\xi}_V : V \in B(T)} $ in $ \PD^* \oplus \R $. Assume
		\begin{equation*}
		\dim \Xi(T) < \#\brac{B(T)}.
		\end{equation*}
		Then there exists $ \eta \in W^m(S)^* $ such that the following hold.
		\begin{enumerate}
			\item For all $ V \in \tree \setminus T $, $ \eta_V = \theta_V \xi_V $ for some $ 0 \leq \theta_V \leq 2 $.
			\item For some $ V \in B(T) $, $ \eta_x \equiv 0 $ for all $ x \in V $.
			\item $ \sum_{x \in S}\xi_x = \sum_{x \in S}\eta_x $ as elements of $ \PD^* $.
			\item $ \sum_{x \in S}\Phi_x(\xi_x) = \sum_{x \in S}\Phi_x(\eta_x) $.
		\end{enumerate}
		Moreover, for such $ \eta $, we have
		\begin{equation}
		\norm{\eta}_{\clu^*} \leq 2\norm{\eta}_{\clu^*}
		\label{eq.equiv-dual-cluster}
		\end{equation}
	\end{lemma}

	\begin{proof}
		We modify the proof of Lemma 6.1 of \cite{BM07}.
		
		Since $ \dim \Xi(T) \leq \#\brac{B(T)} $, $ \set{\bar{\xi}_W, W \in B(T)} $ is not linearly independent, so we may find $ V \in B(T) $ such that
		\begin{equation*}
		\bar{\xi}_V = \sum_{W \in B(T)\setminus V} \lambda_{VW}\cdot \bar{\xi}_W
		\text{ where all } \abs{\lambda_{VW}} \leq 1\,.
		\end{equation*}
		For each $ x \in S $, we set $ \eta_x := \theta_x \cdot \xi_x $, where
		\begin{equation*}
		\theta_x :=
		\begin{cases}
		0&\text{ if } x \in V\,.\\
		1 + \lambda_{VW} &\text{ if } x \in W\text{ and } W \in B(T)\setminus V\,.\\
		1&\text{ otherwise.}
		\end{cases}
		\end{equation*}
		Conditions (1) and (2) then follow by construction.
		
		Now we prove (3) and (4). First we make the following crucial observation. Thanks to our assumption on $\Phi$ and the conditions on the $\lambda_{VW}$'s, we see that
		\begin{equation}
		\Phi_x\brac{(1+\lambda_{VW})\xi_x} = \Phi_x(\xi_x) + \lambda_{VW}\Phi_x(\xi_x).
		\end{equation}{}
		Therefore,
		\begin{equation*}
		\begin{split}
		\sum_{x \in S} \bar{\eta}_x &= \sum_{x \in S}\bar{\xi}_x -\sum_{x \in V}\bar{\xi}_x + \sum_{W \in B(T)\setminus V}\lambda_{VW}\sum_{x \in W}\bar{\xi}_x\\
		&= \sum_{x \in S}\bar{\xi}_x -\brac{ \bar{\xi}_V - \sum_{W \in B(T)\setminus V}\lambda_{VW}\bar{\xi}_W }\\
		&= \sum_{x \in S} \bar{\xi}_x\,.
		\end{split}
		\end{equation*}
		We see that (3) and (4) follow.
		
		Lastly, \eqref{eq.equiv-dual-cluster} follows from \eqref{eq.cluster-dual} and conditions (1) and (3).
		
	\end{proof}

	Let $S \subset \Rn$ and let $\tree$ be a clustering of $S$. For any subset $S' \subset S$, $\tree$ determines a clustering $\tree'$ of $S'$ by restriction.

	The main result of the section is the following.

	\begin{lemma}\label{lem.red-2}
		Let $ k \geq 2 $. Under the hypotheses of Lemma \ref{lem.red-1}, if $ \abs{S} \leq k $, then there exists $ S' \subset S $ satisfying the following.
		\begin{enumerate}
			\item Let $ \tree' $ be the clustering of $ S' $ determined by $ \tree $. For every trunk $ T' $ of $ \tree' $, let $ \Xi(T') $ denote the linear span of $ \set{\bar{\xi}_V : V \in B(T')} $ in $ \PD^* \oplus \R $. Then we have
			\begin{equation*}
			\#\brac{B(T')} \leq \dim \Xi(T').
			\end{equation*}
			\item There exists $ \eta \in W^m(S)^* $ such that the following hold.
			\begin{enumerate}
				\item $ \eta_x $ is a multiple of $ \xi_x $ for each $ x \in S $, and $ \eta_x = 0 $ for $ x \in S \setminus S' $.
				\item $ \norm{\eta}_{W^m(S)^*} \leq C \norm{\xi}_{W^m(S)^*} $, where $ C = C(m,n,k,B) $ with $B$ being an upper bound for $\diam(S)$.
				\item $ \sum_{x \in S}\xi_x  = \sum_{x \in S}\eta_x $.
				\item $ \sum_{x \in S}\Phi_x(\xi_x) = \sum_{x \in S}\Phi_x(\eta_x) $, with $ \Phi_x $ as in Lemma \ref{lem.red-1}.
			\end{enumerate}
		\end{enumerate}
	\end{lemma}
	
	
	\begin{proof}
		Suppose $ S $ itself does not satisfy both of the conclusions. Taking $ \eta = \xi $, we see that $ S $ satisfies (2). Therefore, $ S $ does not satisfy (1). Using condition (2) of Lemma \ref{lem.red-1}, we may shrink $ S $ by one point at a time until conclusions (1), (2a), (2c), and (2d) hold. Meanwhile, (2b) holds, thanks to Lemma \ref{lem.clustering-exist}, Lemma \ref{lem.clust-equiv}, and \eqref{eq.equiv-dual-cluster}.
	\end{proof}

	We will couple Lemma \ref{lem.red-2} with the following result.
	
	\begin{lemma}[Lemma 6.4 of \cite{BM07}]\label{lem.2N-1}
		Let $ S \subset \Rn $ be finite with $ \#(S) \geq 2 $. Let $ \tree $ be a clustering of $ S $. Suppose that for every trunk $ T $ of $ \tree $, $ \abs{\brac{B(T)} }\leq N $ for some $ N \in \N $. Then $ \abs{S} \leq 2^{N-1} $.
	\end{lemma}

	\section{Proof of the main theorem}\label{sec:main proof}
	
	
	We begin the proof of Theorem \ref{thm:main theorem} by showing that one can approximate the convex sets $\Gamma$ arbitrarily well by polytopes, which will allow us to use linear programming. While finer levels of approximation to these convex sets will generally require an arbitrarily increasing number of linear constraints to describe, the constants arising in our proof will be independent of this number.
	
	By a polytope in a finite-dimensional normed vector space $V$, we mean the finite intersection of half-spaces of the form $\{v\in V:\xi(v)\le c\}$, where $\xi\in V^*$ and $c\in\R$.
	
	Let $ v,w $ be two Euclidean vectors. We write $ v \geq w $ if each of the entries of $ v - w $ is nonnegative.
	
	\begin{lemma}\label{lemma:approx convex set}
		Let $V$ be a finite-dimensional normed vector space with norm $\|\cdot\|_V$, and let $K \subset V$ be convex. Given $\delta > 0$, there exists a convex polytope $K_\delta$ such that $K\subset K_\delta\subset B_\delta(K)$, where $B_\delta(K)$ is the $\delta$-neighborhood of $K$ under the metric determined by $\|\cdot\|_V$.
	\end{lemma}
	
	\begin{proof}
		We first address the case where $V=\R^d$, where the norm is the $\ell^\infty$ norm given by $\|(x_1,...,x_d)\|=\max_{1\le j\le d}|x_j|$.
		
		Let $\mathcal{Q}$ be the set of cubes of the form $$Q=[k_1\delta,(k_1+1)\delta]\times...\times[k_d\delta,(k_d+1)\delta],$$ where $k_1,...,k_d\in\Z$. Define
		\begin{equation*}
		K'=\bigcup_{\substack{Q\in\mathcal{Q}\\Q\cap K\neq\emptyset}} Q,
		\end{equation*}
		and let $K''=Conv(K')$, where $Conv(\cdot)$ is used to denote the convex hull of a set. Thus, $K''$ is a convex polytope. By definition, $K\subset K'$.
		
		Let $x\in K''$. Then, there exist $y',z'\in K'$ such that $x$ is on the line segment from $y'$ to $z'$. Since $y',z'\in K'$, there exist $y,z\in K$ such that $\|y-y'\|,\|z-z'\|<\delta$.
		
		Consider the function $f(t)=\|t(y'-y)+(1-t)(z'-z)\|$. Then $f(0),f(1)<\delta$ and $f$ is a convex, nonnegative function, so $f(t)<\delta$ for all $t\in[0,1]$. Pick $t_0\in[0,1]$ such that $x=t_0y'+(1-t_0)z'$. Then, $f(t_0)<\delta$ means that $\|x-[ty+(1-t)y']\|<\delta$. Since $K$ is convex, $ty+(1-t)y'\in K$, so $x$ is within distance $\delta$ of $K$. Thus, $K'\subset B_\delta(K)$, completing the proof in the case $V=\R^d$.
		
		Now suppose that $V$ is an arbitrary $d$-dimensional, normed space. Since any two norms on a finite-dimensional space are equivalent, there exists $M<\infty$ such that $M^{-1}\|v\|_V\le\|T^{-1}(v)\|_{\ell^\infty(\R^d)} \le M\|v\|_V$.
		
		Let $T:\R^d\to V$ be a linear isomorphism and let $\tilde{K}\subset\R^n$ be a polytope satisfying $T^{-1}(K)\subset \tilde{K}\subset B_\epsilon(T^{-1}(K))$, where $\epsilon>0$ is to be determined. It follows that $K\subset T(\tilde{K})\subset B_{M\epsilon}(K)$ and that $T(\tilde{K})$ is a polytope in $V$. (To see the latter, observe that for linear functionals $\xi$ on $\R^d$, $\xi(v)\le c$ if and only if $\xi\circ T^{-1}(T(v))\le c$ and $\xi\circ T^{-1}\in V^*$.)
		
		Thus, choosing $\epsilon=\delta/M$, we see that $T(\tilde{K})$ is the desired polytope.
	\end{proof}

	
	\subsection{Theorem \ref{thm:main theorem} with $ \mathbb{X} = \V $}

	\begin{proof}[Proof of Theorem \ref{thm:main theorem}  with $ \mathbb{X} = \V $]
		Given $E\subset \R^n$ and $K(x)\subset\vec{\P}$ for each $x\in E$, we define
		\begin{equation}
		\|(K(x))_{x\in E}\|_{W^m(E)}:=\inf\{\|(\vec{P}^x)_{x\in E}\|_{W^m(E)}:\vec{P}^x\in K(x)\text{ for all } x\in E\}.
		\label{eq:selection norm}
		\end{equation}
		
		While not strictly a norm, the above notation allows for a concise description of a quantity which is the main focus of the proof.
		
		Our goal is to show there exists $C=C(m,n,D,B)$ such that for any finite $S\subset \R^n$ satisfying $|S|\le B$, there exists $S'\subset S$ with $|S'|\leq2^{\dim\vec{\mathcal{P}}}$ such that
		\begin{equation*}
		C^{-1}\|(G(x))_{x\in S'}\|_{W^m(S')}\leq \|(G(x))_{x\in S}\|_{W^m(S)}\leq C\|(G(x))_{x\in S'}\|_{W^m(S')}.
		\end{equation*}
		
		If so, it follows that
		\begin{equation*}
		\|(G(x))_{x\in S'}\|_{W^m(S')}\leq 1 \text{ for all }S'\subset S \text{ satisfying }|S'|=2^{\dim\vec{\mathcal{P}}}
		\end{equation*}
		implies
		\begin{equation*}
		\|(G(x))_{x\in S}\|_{W^m(S)}\leq C
		\end{equation*}
		for all $S\subset E$ satisfying $|S|=k^\sharp$.
		
		We now make the following reduction: it suffices to prove Theorem \ref{thm:main theorem} in the case that each $G(x)$ is a polytope.
		
		If not, replace each $G(x)$ with $G(x)_\delta$ for sufficiently small $\delta>0$, where $G(x)_\delta$ is the polytope guaranteed by Lemma \ref{lemma:approx convex set}. By taking $\delta>0$ small enough, one may approximate both $\|(G(x))_{x\in S}\|_{W^m(S)}$ and $\|(G(x))_{x\in S'}\|_{W^m(S')}$ within a factor of 2, as these norms are continuous with respect to the relevant metrics.
		
		To this end, we replace each $G(x)$ with $G(x)_\delta$, which will now be denoted $K_x$, as $\delta$ is fixed. For each $x$, we write $K_x=\{\vec{P}:\Omega_x \vec{P}\le \vec{c}_x\}$ for some linear map $\Omega_x:\vec{\P}\to\R^{m_x}$, where $m_x\in\N$. We will occasionally write $\Omega:W^m(S)\to \prod_x\R^{m_x}$ to denote the mapping which sends $(\vec{P}^x)_{x\in S}$ to $(\Omega_x\vec{P}^x)_{x\in S}$.
		
		We begin by writing $\|(K_x)_{x\in S}\|_{W^m(S)}$ as the solution to a linear programming problem:
		\begin{align}\label{eq:approx norms}
		\|(K_x)_{x\in S}\|_{W^m(S)}&=\inf_{(\Omega_x\vec{P}^x\le \vec{c}_x)_{x\in S}} \|(\vec{P}^x)_{x\in S}\|_{W^m(S)}\\
		&=\inf_{(\Omega_x\vec{P}^x\le \vec{c}_x)_{x\in S}} \sup_{\|(\xi_x)_{x\in S}\|_{W^m(S)^*}\le1} (\xi_x)_{x\in S}[(\vec{P}^x)_{x\in S}].
		\end{align}
		
		By Lemma \ref{lemma:approx convex set}, the unit ball in $W^m(S)^*$ may be approximated within a factor of 2 by a polytope, written as $\{(\xi_x)_{x\in S}:L(\xi_x)_{x\in S}\le\one_k$ for some $k\in \N$ and linear map $L:W^m(S)^*\to\R^k$. Thus, we may rewrite \eqref{eq:approx norms} as
		
		\begin{equation}
		\|(K_x)_{x\in S}\|_{W^m(S)}\approx \inf_{(\Omega_x\vec{P}^x\le \vec{c}_x)_{x\in S}} \sup_{L(\xi_x)_{x\in S}\le\one_k} (\xi_x)_{x\in S}[(\vec{P}^x)_{x\in S}]
		\label{eq.4.4}
		\end{equation}
		for some linear map $L:W^m(S)^*\to\R^k$ and some $k\in\N$.
		
		The advantage of this formulation is that it becomes possible to apply the LP Duality Theorem (Lemma \ref{lem.LP-abs} in Appendix) to the supremum above, giving us
		
		\begin{align*}
		\|(K_x)_{x\in S}\|_{W^m(S)}&\approx \inf_{(\Omega_x\vec{P}^x\le \vec{c}_x)_{x\in S}} \inf_{\substack{y\ge0\\L^Ty=(\vec{P}^x)_{x\in S}}} \one_k\cdot y\\
		&=\inf_{\substack{(\Omega_x\vec{P}^x\le c_x)_{x\in S}\\y\ge0\\L^Ty=(\vec{P}^x)_{x\in S}}} \one_k\cdot y\\
		&=\inf_{\substack{-\Omega L^Ty\ge -c\\y\ge0}} \one_k\cdot y.
		\end{align*}
		Note the referenced linear programming problem is feasible, as its solution corresponds to finding the smallest norm of a vector in a closed set.
		
		Applying the Duality Theorem again, one obtains
		
		\begin{align}\label{eq:big finish}
		\|(K_x)_{x\in S}\|_{W^m(S)}&\approx \sup_{\substack{(z_x\ge0)_{x\in S}\\L(-\Omega^T)z\le\one_k}}\sum_x -\vec{c}_x\cdot z_x\\
		&=\sup_{\substack{L(\xi_x)_{x\in S}\le\one_k\\(\xi_x\in(-\Omega_x^T)\R_+^{m_x})_{x\in S}}}\sup_{\substack{(z_x\ge0)_{x\in S}\\(-\Omega_x^Tz_x=\xi_x)_{x\in S}}}\sum_x -\vec{c}_x\cdot z_x\\
		&\approx\sup_{(\xi_x\in(-\Omega_x^T)\R_+^{m_x})_{x\in S}} \frac{\sum_{x\in S}f_x(\xi_x)}{\|(\xi_x)_{x\in S}\|_{W^m(S)^*}},
		\end{align}
		where
		\begin{equation*}
		f_x(\xi_x)=\sup_{\substack{z\ge0\\-\Omega_x^Tz=\xi_x}}\sum_x -\vec{c}_x\cdot z.
		\end{equation*}
		
		
		Fix $(\xi_x)_{x\in S}$ such that $\xi_x\in (-\Omega_x^T)\R_+^{m_x}$ for all $x\in S$. We see that $f$ satisfies the hypotheses of Lemmas \ref{lem.red-1} and \ref{lem.red-2}. By Lemma \ref{lem.2N-1}, we may apply Lemma \ref{lem.red-2} repeatedly until the $S'$ in the conclusion satisfies $|S'|\le 2^{\dim\P}$.

		Let $(\eta_x)_{x\in S}$ be as guaranteed in the conclusion of the lemma and recall $S'=\{x\in S: \eta_x\neq0\}$. Thus,
		\begin{equation*}
		\|(\eta_x)_{x\in S}\|_{W^m(S')^*}=\|(\eta_x)_{x\in S}\|_{W^m(S)^*}\lesssim\|(\xi_x)_{x\in S}\|_{W^m(S)^*}
		\end{equation*}
		and $|S'|\le 2^{\dim\vec{\P}}$. Note that each $\eta_x$ is obtained by multiplying some $\xi_x$ by a nonnegative scalar; thus, $\eta_x\in (-\Omega_x^T)\R_+^{m_x}$ for all $x\in S$.

		By this reasoning and \eqref{eq:big finish} applied both as written above and with $S'$ in place of $S$,
		\begin{align*}
		\|(K_x)_{x\in S}\|_{W^m(S)}&\approx\sup_{(\xi_x\in\Omega_x^T\R_+^{m_x})_{x\in S}} \frac{\sum_{x\in S}f_x(\xi_x)}{\|(\xi_x)_{x\in S}\|_{W^m(S)^*}}\\
		&\lesssim\sup_{(\eta_x\in\Omega_x^T\R_+^{m_x})_{x\in S'}} \frac{\sum_{x\in S}f_x(\eta_x)}{\|(\eta_x)_{x\in S}\|_{W^m(S')^*}}\\
		&\approx\|(K_x)_{x\in S'}\|_{W^m(S')}.
		\end{align*}
		
	\end{proof}

	\subsection{Theorem \ref{thm:main theorem} with $ \mathbb{X} = \hV $}
	
	\newcommand{\hW}{\dot{W}^m(S)}
	\newcommand{\LS}{H(S)}
	
	In this section, we point out the modifications needed in order to prove Theorem \ref{thm:main theorem} for the case $ \mathbb{X} = \hV $.
	
	Let $ S \subset \Rn $ be a finite set. Recall the definition of $ \norm{\cdot}_{\hW} $ in \eqref{eq.W-2}. We define
	\begin{equation*}
	H(S):= \mathrm{span}\,\set{\xi_{\alpha,y,z,j}: y,z \in S, \,z \neq y, \,\abs{\alpha} \leq m-1,\,1 \leq j \leq D} \,,
	\end{equation*}
	where each $\xi_{\alpha,y,z,j} \in W^m(S)^*$ is characterized by the action
	\begin{equation*}
	\xi_{\alpha,y,z,j}[(\vP^x)_{x \in S}] = \frac{\d^\alpha(P^y_j - P^z_j)(y)}{\abs{y-z}^{m-\abs{\alpha}}}\,.
	\end{equation*}
	Then the norm $\norm{\cdot}_{\hW}$ can be computed via the formula
	\begin{equation*}
		\norm{(\vP^x)_{x \in S}}_{\hW}
		= \sup_{\substack{\xi \in \LS\\\norm{\xi}_{W^m(S)^*} \leq 1}}
		\xi [ (\vP^x)_{x \in S} ]\,.
    \end{equation*}

	
	Mirroring \eqref{eq:selection norm}, we define the selection ``seminorm'' to be
		\begin{equation*}
		\norm{(K_x)_{x \in S}}_{\dot{W}^m(S)} := \inf\{\|(\vec{P}^x)_{x\in S}\|_{\dot{W}^m(S)}:\vec{P}^x\in K(x)\text{ for all } x\in S\}.
		\end{equation*}
		We repeat proof of Theorem \ref{thm:main theorem} with $ \mathbb{X} = \V $ in the previous section, but with the following modifications.
		\begin{itemize}
			\item We use $ \norm{\cdot}_{\hW}$ in place of $ \norm{\cdot}_{W^m(S)} $ (both the Whitney seminorm and the selection ``seminorm'').
			\item All the linear functionals will be chosen from $ \LS \subset W^m(S)^* $.
			\item The map $ L $ in \eqref{eq.4.4} will be replaced by a suitable linear map $ \tilde{L}:\LS \to \R^{\tilde{k}} $ for some $ \tilde{k} \in \N $.
		\end{itemize}

    This concludes all the necessary modifications for the proof of Theorem \ref{thm:main theorem} with $ \mathbb{X} = \hV $.
    
    The proof of Theorem \ref{thm:main theorem} is complete.

	\section{Vector-Valued Shape Fields Finiteness Principle}

	In this section we use what is colloquially known as the ``gradient trick'' to prove Theorem \ref{thm:SF FP} using the $D=1$ case proven in \cite{FIL16}. (See \cite{FIL16, FL14}.)
	
	The following proof will require working in both $\R^n$ and $\R^{n+D}$, so we provide a brief introduction to some of the notation.
	
	The variable for $\R^n$ will be $x$, while $\R^{n+D}$ will be viewed as $\{z=(x,\xi):x\in\R^n,\xi\in\R^D\}$. The appropriate level of regularity for $\R^{n+D}$ will be $C^{m+1}$, so let $\P^+$ denote the vector space of $\R$-valued, $m$-degree polynomials over $\R^{n+D}$. (Recall that $\vec{\P}$ is the vector space of $\R^D$-valued, $(m-1)$-degree polynomials over $\R^n$.)
	
	\begin{proof}[Proof of Theorem \ref{thm:SF FP}]
		Let $E, Q_0\subset\R^n, (\vec{\Gamma}(x,M))_{x\in E, M>0}, C_w, 0<\delta_{Q_0}\le\delta_{\max}, x_0\in E\cap 5Q_0$ as in the hypotheses of Theorem \ref{thm:SF FP} be given.
		
		Let $E^+=\{(x,0):x\in E\}\subset\R^{n+D}$. For $(x_0,0)\in E^+$, define
		\begin{equation}
		\Gamma((x_0,0),M)=\{P\in\P^+:P(x,0)=0, \nabla_\xi P(x,0)\in \vec{\Gamma}(x_0,M)\}.
		\end{equation}
		
		We now show that $(\Gamma(z,M))_{z\in E^+}$ satisfies the hypotheses of the $D=1$ case of Theorem \ref{thm:SF FP}.
		
		Let $S^+\subset E^+$ with $|S^+|\le k^\sharp$. By definition, $S^+$ is of the form $\{(x,0):x\in S\}$ for some $S\subset E$ with $|S|\le k^\sharp$.
		
		By hypothesis of Theorem \ref{thm:SF FP}, there exist $(\vec{P}^x)_{x\in S}$ such that
		\begin{equation}
		\|(\vec{P}^x)_{x\in S}\|_{\dot{W}^m(S)}\le M_0.
		\end{equation}
		and
		\begin{equation}
		\vec{P}^x\in \vec{\Gamma}(x,M_0)\text{ for all }x\in S.
		\end{equation}
		
		For $z=(x_0,0)\in E^+$, define
		\begin{equation}
		P^z(x,\xi)=P^{(x_0,0)}(x,\xi):=\sum_{j=1}^D \xi_jP_j(x).
		\end{equation}

		Clearly, $P^{(x_0,0)}(x_0,0)=0$ and $\nabla_\xi P^{(x_0,0)}=\vec{P}^{x_0}$, so $P^z\in \Gamma(z,M_0)$ for all $z\in E^+$.
		
		Let $(x_0,0),(y_0,0)\in E^+$. Then,
		\begin{equation}
		\d_x^\alpha P^{(x_0,0)}(x,0)=0\text{ and }\d_x^\alpha\d^\beta_\xi P^{(x_0,0)}(x,0)=0\text{ for }|\beta|\ge2
		\end{equation}
		by definition, and for $1\le j\le D$,
		\begin{align}
		\left|\d_x^\alpha\d_{\xi_j}\left(P^{(x_0,0)}-P^{(y_0,0)}\right)(x_0,0)\right|&=\left|\d_x^\alpha(P_j^{x_0}-P_j^{y_0})(x_0,0)\right|\\
		&\leq C|x_0-y_0|^{m-|\alpha|}\\
		&=C|(x_0,0)-(y_0,0)|^{(m+1)-(|\alpha|+1)}.
		\end{align}
		
		Thus, $(P^z)_{z\in S^+}$ satisfy \eqref{eq:SF homogeneity condition}.
		
		To demonstrate $(C_w,\delta_{\max})$-convexity, let $0<\delta\le\delta_{\max}, x\in S^+, M<\infty, P_1, P_2, Q_1, Q_2\in \mathcal{P}^+$ be as in Definition \ref{def.shape field}. If $P:=Q_1\odot_{(x_0,0)}Q_1\odot_{(x_0,0)}P_1+Q_2\odot_{(x_0,0)}Q_2\odot_{(x_0,0)}P_2$, then $P(x_0,0)=0$ and
		\begin{equation}
		\nabla_\xi P(x,0)=[Q_1\odot_{(x_0,0)}Q_1\odot_{(x_0,0)}\nabla_\xi P_1](x,0)+[Q_2\odot_{(x_0,0)}Q_2\odot_{(x_0,0)}\nabla_\xi P_2](x,0),
		\end{equation}
		which lies in $\Gamma(x,C_w M)$ by the $(C_w,\delta_{\max})$-convexity of the $\vec{\Gamma}(x,M)$.
		
		Let $Q'$ be the unit cube in $\R^D$. By the $D=1$ case of Theorem \ref{thm:SF FP} applied to $E^+\subset\R^{n+D}, (\Gamma(z,M))_{z\in E+, M>0}, (x_0,0), Q_0\times Q'$, we have the following. There exist $F\in C^{m+1}(\R^{n+D},\R)$ and $P^0\in\Gamma((x_0,0),CM_0)$ such that
		\begin{equation}\label{eq:conc1}
		J_{(x,0)}F\in\Gamma((x,0),CM)\text{ for all }(x,0)\in E^+;
		\end{equation}
		
		\begin{equation}\label{eq:conc2}
		|\d_x^\alpha\d_\xi^\beta(F-P^0)(x,\xi)|\le CM_0\text{ for all }(x,\xi)\in Q_0\times Q', |\alpha|+|\beta|\le m+1;
		\end{equation}
		and
		\begin{equation}\label{eq:conc3}
		\text{In particular, }|\d_x^\alpha\d_\xi^\beta F(x,\xi)|\le CM_0\text{ for }|\alpha|+|\beta|=m+1.
		\end{equation}
		
		Define $\vec{G}(x):=\nabla_\xi F(x,0)$ and $\vec{Q}^0(x)=\nabla_\xi P^0(x,0)$. We claim $\vec{G}\in C^m(\R^n,\R^D)$ and $\vec{Q}^0\in\vec{\Gamma}(x,CM)$ are the desired function and jet, respectively, found in the conclusion of Theorem \ref{thm:SF FP}.
		
		First, by \eqref{eq:conc1},
		\begin{equation}
		J_xG(y)=\nabla_\xi J_{(y,0)}F(x,0)\in \vec{\Gamma}(x,CM)
		\end{equation}
		because $J_{(y,0)}F(x,0)\in\vec{\Gamma}((x,0),CM)$.
		
		Next, for any $|\alpha|\le m$ and $1\le j\le D$,
		\begin{align}
		|\d_x^\alpha(G_j-Q^0_j)(x)|&=|\d_x^\alpha(\d_{\xi_j}F-\d_{\xi_j}P^0)(x,0)|\\
		&\le CM_0\delta_{Q_0}^{(m+1)-(|\alpha|+1)}=CM_0\delta_{Q_0}^{m-|\alpha|}
		\end{align}
		by \eqref{eq:conc2}.
		
		Lastly, for $|\alpha|=m$,
		\begin{equation}
		|\d_x^\alpha G_j(x)|=|\d_x^\alpha \d_{\xi_j}F|\le CM_0
		\end{equation}
		via \eqref{eq:conc3}.
	\end{proof}

	\appendix

	\section{Linear programming and duality}

	\begin{lemma}[LP Duality Theorem]\label{lem.LP}
		Let $ p,q $ be positive integers. Let $ c \in \R^p $ and $ b \in \R^q $. Let $ A : \R^p \to \R^q $ be a linear map. Consider the following two optimization problems.
		\begin{eqnarray}
		&\text{Maximize } c^T\cdot x
		\text{ subject to }
		Ax \leq b\,.
		\label{eq.dual-max}\\
		&\text{Minimize } b^T\cdot y
		\text{ subject to }
		A^Ty = c
		\text{ and }
		y \geq 0\,.
		\label{eq.dual-min}
		\end{eqnarray}

		Suppose one of \eqref{eq.dual-max} or \eqref{eq.dual-min} has a feasible solution, then both have feasible and optimal solutions. Moreover, if $ x_0 $ optimizes \eqref{eq.dual-max} and $ y_0 $ optimizes \eqref{eq.dual-min}, then $ c^T\cdot x_0 = b^T \cdot y_0 $, i.e. the maximum of \eqref{eq.dual-max} equals the minimum of \eqref{eq.dual-min}.
		
		The same conclusion holds if we replace ``$ Ax \leq b $'' by ``$ Ax \leq b $ and $ x \geq 0 $'' in \eqref{eq.dual-max} and ``$ A^Ty = c $'' by ``$ A^Ty \geq c $'' in \eqref{eq.dual-min}.
	\end{lemma}

	See \cite{M07-LP} for a proof.

	\newcommand{\pb}[1]{\langle#1\rangle}
	\newcommand{\euc}{\mathbf{E}}

	We generalize the theorem above to finite dimensional normed spaces.

	\begin{lemma}\label{lem.LP-abs}
		Let $ V $ be a finite-dimensional normed vector space with norm $ \norm{\cdot}_V $ and dual $ V^* $. Let $ L : V^* \to \R^q $ be a linear map and let $ L^*: \R^q \to V $ be the dual operator of $ L $ defined by\footnote{Here we identify the dual of any Euclidean space with itself via the dot product.}
		\begin{equation*}
		x^T\cdot L(\phi) = \pb{\phi,L^* x}
		\text{ for all } x \in \R^q \text{ and } \phi \in V^*.
		\end{equation*}
		Let $ b \in \R^q $. Suppose there exists $ \phi_0 \in V^* $ such that $ L(\phi_0) \leq b $. Then
		\begin{equation}
		\sup_{L(\phi) \leq b} \pb{\phi,v} = \inf_{\substack{y \geq 0\\L^*y = v}} b^T\cdot y.
		\label{eq.lem-LP}
		\end{equation}
	\end{lemma}
	
	\begin{proof}
		Let $ p = \dim V < \infty $. There exists a linear isomorphism $ J: V \to \R^p $. Let $ J^* : \R^p \to V^* $ denote its dual. Note that $ J^* $ is also a linear isomorphism. We have the following diagram.
		\begin{equation*}
		\begin{tikzcd}
		\R^p \arrow{d} & V \arrow[l,"J"] & \R^q \arrow[l,"L^*"]\arrow{d} \\
		\R^p \arrow[r,"J^*"] & V^*\arrow[r,"L"] & \R^q
		\end{tikzcd}
		\end{equation*}
		For each $ v \in V $ and $ \phi \in V^* $, there exist unique $ c,x\in \R^p $ such that $ J^{-1}(p) = v $ and $ J^*(x) = \phi $. Thus, thanks to LP Duality Theorem (see Appendix), we have
		\begin{equation*}
		\begin{split}
		\sup_{L(\phi)\leq b}\pb{\phi,v} &= \sup_{L\circ J^*(x) \leq b}\pb{J^*(x), J^{-1}(c)} \\
		&= \sup_{L\circ J^*(x)\leq b}c^T\cdot x\\
		&= \inf_{\substack{(L\circ J^*)^Ty = c\\y\geq 0}}b^T \cdot y.
		\end{split}
		\end{equation*}
		Notice that $ (L\circ J^*)^T = J\circ L^* $. Moreover, since $ J $ is an isomorphism, the equality $ J \circ L^* y = c $ is equivalent to $ L^*y = J^{-1}(c) = v $. \eqref{eq.lem-LP} follows.
	\end{proof}

\bibliography{Whitney-bib}{}
\bibliographystyle{plain}

\end{document}